\documentclass[preprint,12pt,3p]{elsarticle}

\usepackage{amssymb}

\usepackage{amsthm}
\usepackage{graphicx}
\usepackage{amsmath}
\usepackage{hyperref}
\usepackage{graphicx}

\usepackage{pgf,tikz}

\usetikzlibrary{snakes}
\tikzstyle{printersafe}=[snake=snake,segment amplitude=0 pt]

\usetikzlibrary{arrows}
\usetikzlibrary{shapes}
\usepackage{multirow}
\usepackage{latexsym}
\usepackage{setspace}
\PassOptionsToPackage{boxed,section}{algorithm}
\usepackage{algorithm}
\usepackage{algorithmic}
\usepackage{enumerate}
\usepackage{amsmath}			
\usepackage{amssymb}			
\usepackage{url}
\usepackage{setspace}
\usepackage{tikz}
\usetikzlibrary{arrows}
\usetikzlibrary{shapes}
\usetikzlibrary{decorations.markings}
\usepackage{geometry}
\usepackage{bbm}

\newtheorem{proposition}{\em Proposition}
\newtheorem{problem}{\em Problem}
\newtheorem{theorem}{\em Theorem}
\newtheorem{conjecture}{\em Conjecture}
\newtheorem{definition}{\em Definition}
\newtheorem{lemma}{\em Lemma}
\newtheorem{remark}{\em Remark}
\newtheorem{corollary}{\em Corollary}
\newtheorem{observation}{\em Observation}

\journal{Latex Templates}

\begin{document}

\begin{frontmatter}

\title{Non-conflicting no-where zero $Z_2\times Z_2$ flows in cubic graphs}

\author[label1]{Vahan Mkrtchyan}
\address[label1]{Department of Mathematics and Computer Science,\\ College of the Holy Cross, Worcester, MA, USA}

%
\ead{vahan.mkrtchyan@gssi.it}
%
%

\begin{abstract}
Let $Z_2\times Z_2=\{0, \alpha, \beta, \alpha+\beta\}$. If $G$ is a bridgeless cubic graph, $F$ is a perfect matching of $G$ and $\overline{F}$ is the complementary 2-factor of $F$, then a no-where zero $Z_2\times Z_2$-flow $\theta$ of $G/\overline{F}$ is called non-conflicting with respect to $\overline{F}$, if $\overline{F}$ contains no edge $e=uv$, such that $u$ is incident to an edge with $\theta$-value $\alpha$ and $v$ is incident to an edge with $\theta$-value $\beta$. In this paper, we demonstrate the usefulness of non-conflicting flows by showing that if a cubic graph $G$ admits such a flow with respect to some perfect matching $F$, then $G$ admits a normal 6-edge-coloring. We use this observation in order to show that claw-free bridgeless cubic graphs, bridgeless cubic graphs possessing a 2-factor having at most two cycles admit a normal 6-edge-coloring. We demonstrate the usefulness of non-conflicting flows further by relating them to a recent conjecture of Thomassen about edge-disjoint perfect matchings in highly connected regular graphs. In the end of the paper, we construct infinitely many 2-edge-connected cubic graphs such that $G/\overline{F}$ does not admit a non-conflicting no-where zero $Z_2\times Z_2$-flow with respect to any perfect matching $F$.
\end{abstract}


\end{frontmatter}



\section{Introduction}
\label{IntroSection}

Graphs considered in this paper are finite and undirected. They do not contain loops, though they may contain parallel edges. We also consider pseudo-graphs, which may contain both loops and parallel edges, and simple graphs, which contain neither loops nor parallel edges. As usual, a loop contributes to the degree of a vertex by two. For a graph $G$, $V=V(G)$ and $E=E(G)$ will denote the sets of vertices and edges of $G$, respectively. A matching in a graph $G$ is a subset $F$ of edges such that no two edges of $F$ share a vertex. A matching $F$ is perfect if every vertex of the graph is incident to an edge from $F$. For $k\geq 1$, a $k$-factor of a graph $G$ is a spanning $k$-regular subgraph of $G$. Note that if $K$ is a 1-factor of $G$, then $E(K)$ is a perfect matching in $G$.

A graph $G$ is $k$-regular if every vertex of $G$ is of degree $k$. A graph is cubic if it is 3-regular. Note that if $G$ is a cubic graph then $F$ is a perfect matching in $G$ if and only if $G-F$ is a 2-factor in $G$. This 2-factor will be called a complementary 2-factor of $F$ in $G$. For a perfect matching $F$ of a cubic graph $G$, its complementary 2-factor will be denoted by $\overline{F}$.

For a graph $G$ and a vertex $v$ let $\partial_{G}(v)$ be the set of edges of $G$ that are incident to $v$ in $G$. If $G$ is a graph then its girth is the length of the shortest cycle in $G$. For $n\geq 1$ let $K_n$ denote the unique graph on $n$ vertices where every pair of vertices is an edge in it. Such graph is called complete and usually is denoted by $K_n$ (Figure \ref{fig:K4}). A graph $G$ is bipartite, if $V(G)$ can be partitioned into two sets $V_1$ and $V_2$, such that every edge of $G$ joins a vertex from $V_1$ to $V_2$. A bipartite graph is called complete if every vertex of $V_1$ is joined to every vertex of $V_2$. When $G$ is a complete bipartite graph with $|V_1|=m$ and $|V_2|=n$, then it will be denoted by $K_{m,n}$.

For $k\geq 1$ a graph $G$ is called cyclically $k$-edge-connected, if we have to delete at least $k$ edges of $G$ so that the resulting graph contains at least two components containing a cycle. A simple path of a graph $G$ is called Hamiltonian if all vertices of $G$ lie on it. Similarly, a simple cycle of a graph $G$ is called Hamiltonian if all vertices of $G$ lie on it. A graph $G$ is called Hamiltonian if it contains a Hamiltonian cycle.

A $k$-edge-coloring of a graph $G$ is an assignment of colors $\{1,...,k\}$ to edges of $G$, such that adjacent edges receive different colors \cite{stiebitz:2012}. The smallest $k$ for which $G$ admits a $k$-edge-coloring, is called the chromatic index and is denoted by $\chi'(G)$. Vizing's classical theorem in the area states that if $G$ is a simple graph, then $\Delta(G)\leq \chi'(G)\leq \Delta(G)+1$ \cite{vizing:1964}. Here $\Delta(G)$ denotes the maximum degree of a vertex in $G$. Note that for cubic graphs this becomes to $3\leq \chi'(G)\leq 4$. Holyer's theorem states that the problem of testing a given cubic graph for $\chi'(G)=3$ is an NP-complete problem \cite{holyer:1981}.

Let $G$ and $H$ be two cubic graphs. If there is a mapping $\phi:E(G)\rightarrow E(H)$, such that for each $v\in V(G)$ there is $w\in V(H)$ such that $\phi(\partial_{G}(v)) = \partial_{H}(w)$, then $\phi$ is called an $H$-coloring of $G$. If $G$ admits an $H$-coloring, then we will write $H
\prec G$. It can be easily seen that if $H\prec G$ and $K\prec H$, then $K\prec G$. In other words, $\prec$ is a transitive relation defined on the set of cubic graphs. If $G$ is the complete bipartite graph $K_{3,3}$ and $H$ is the complete graph $K_4$, then Figure \ref{fig:HcoloringG} shows an example of an $H$-coloring of $G$. Here $V(H)=\{1,2,3,4\}$ and $E(H)=\{a_1, a_2, a_3, a_4, a_5, a_6\}$. Figure \ref{fig:HcoloringG} shows the colors of edges of $G$ with the edges of $H$, and the labels of vertices of $G$ are the vertices of $K_4$ that this vertex is mapped by the $H$-coloring of $G$.

\begin{figure}[ht]
		\begin{center}
			\begin{tikzpicture}[scale=0.85]
			

\def \r {0.2}
\def \radius {\r cm}
\def \c {\r}

\node at (-1,1) {$H$};

\draw (0,0) circle (\radius);
\node at (0,0) {$1$};

\draw (0,2) circle (\radius);
\node at (0,2) {$2$};

\draw (2,2) circle (\radius);
\node at (2,2) {$3$};

\draw (2,0) circle (\radius);
\node at (2,0) {$4$};

\draw[-, thick] (0,\r)--(0,2-\r); 
\node at (-\r,1) {$a_1$};

\draw[-, thick] (\r,2)--(2-\r,2);
\node at (1,2+\r) {$a_2$};

\draw[-, thick] (2,2-\r)--(2,\r);
\node at (2+\r,1) {$a_3$};

\draw[-, thick] (2-\r,0)--(\r,0);
\node at (1,-\r) {$a_4$};

\draw[-, thick] (\r/2,\r/2)--(2-\r/2,2-\r/2);
\node at (2*\r,2*\r+\c) {$a_5$};
\node at (2-2*\r,2-2*\r-\c) {$a_5$};

\draw[-, thick] (2-\r/2,\r/2)--(\r/2,2-\r/2);
\node at (2-2*\r,2*\r-\c) {$a_6$};
\node at (2*\r,2-2*\r+\c) {$a_6$};

\node at (9,1) {$G$};

\draw (4,0) circle (\radius);
\node at (4,0) {$1$};

\draw (8,0) circle (\radius);
\node at (8,0) {$3$};

\draw (6,-2) circle (\radius);
\node at (6,-2) {$4$};

\draw (4,2) circle (\radius);
\node at (4,2) {$1$};

\draw (8,2) circle (\radius);
\node at (8,2) {$3$};

\draw (6,4) circle (\radius);
\node at (6,4) {$4$};


\draw[-, thick] (4,\r)--(4,2-\r);
\node at (4-\r,1) {$a_1$};

\draw[-, thick] (8,\r)--(8,2-\r);
\node at (8+\r,1) {$a_2$};

\draw[-, thick] (4+\r/2,-\r/2)--(6-\r/2,-2+\r/2);
\node at (5-\c,-1-\c) {$a_4$};

\draw[-, thick] (8-\r/2,-\r/2)--(6+\r/2,-2+\r/2);
\node at (7+\c,-1-\c) {$a_3$};

\draw[-, thick] (4+\r/2,2+\r/2)--(6-\r/2,4-\r/2);
\node at (5-\c,3+\c) {$a_4$};

\draw[-, thick] (8-\r/2,2+\r/2)--(6+\r/2,4-\r/2);
\node at (7+\c,3+\c) {$a_3$};


\draw[-, thick] (6,-2+\r)--(6,4-\r);
\node at (6+\c,-2+2*\r+\c) {$a_6$};
\node at (6-\c,4-2*\r-\c) {$a_6$};

\draw[-, thick] (4+\r/2,\r/2)--(8-\r/2,2-\r/2);
\node at (4+\r/2+\c,\r+\c) {$a_5$};
\node at (8-\r/2-\c,2-\r-\c) {$a_5$};

\draw[-, thick] (4+\r/2,2-\r/2)--(8-\r/2,\r/2);
\node at (4+\r/2+\c,2-\r-\c) {$a_5$};
\node at (8-\r/2-\c,\r+\c) {$a_5$};

			\end{tikzpicture}

		\end{center}
		
		\caption{An example of an $H$-coloring of $G$.}\label{fig:HcoloringG}
	\end{figure}
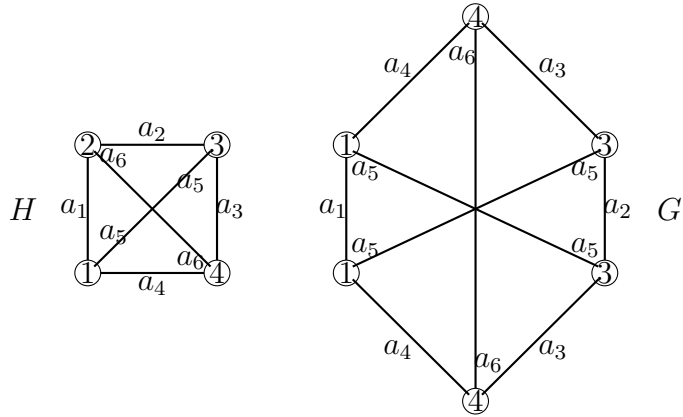

Let $P_{10}$ be the well-known Petersen graph (Figure \ref{fig:Petersen10}). The main topic of this paper is the Petersen Coloring Conjecture of Jaeger. It is a striking conjecture in graph theory that asserts that the edge-set of every bridgeless cubic graph $G$ can be colored by using as set of colors the edge-set of the Petersen graph $P_{10}$ in such a way that adjacent edges of $G$ receive as colors adjacent edges of $P_{10}$.

    \begin{figure}[ht]
	\begin{center}
	\begin{tikzpicture}[style=thick]
\draw (18:2cm) -- (90:2cm) -- (162:2cm) -- (234:2cm) --
(306:2cm) -- cycle;
\draw (18:1cm) -- (162:1cm) -- (306:1cm) -- (90:1cm) --
(234:1cm) -- cycle;
\foreach \x in {18,90,162,234,306}{
\draw (\x:1cm) -- (\x:2cm);
\draw[fill=black] (\x:2cm) circle (2pt);
\draw[fill=black] (\x:1cm) circle (2pt);
}
\end{tikzpicture}
	\end{center}
	\caption{The graph $P_{10}$.}\label{fig:Petersen10}
\end{figure}
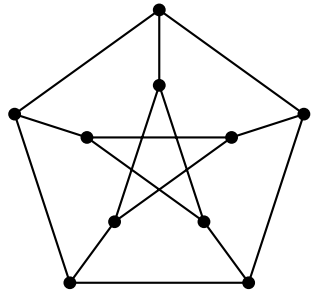

\begin{conjecture}\label{conj:P10conj} (Jaeger, 1988 \cite{Jaeger1988}) For any bridgeless cubic
graph $G$, one has $P_{10} \prec G$.
\end{conjecture}
The conjecture is well-known and it is largely considered hard to prove since it implies some other classical conjectures in the field such as Berge-Fulkerson Conjecture (Conjecture \ref{conj:BergeFulkerson} below), Cycle Double Cover Conjecture, (5, 2)-cycle-cover conjecture (Conjecture \ref{conj:52CDC} below) and the Shortest Cycle Cover Conjecture (see \cite{Fulkerson,Jaeger1985,Zhang1997}).
\begin{conjecture}\label{conj:BergeFulkerson} (Berge-Fulkerson, 1972 \cite{Fulkerson,Seymour}) Any bridgeless
cubic graph $G$ contains six (not necessarily distinct) perfect matchings
$F_1, \ldots , F_6$ such that any edge of $G$ belongs to exactly two of them.
\end{conjecture}

\begin{conjecture}\label{conj:52CDC}
((5, 2)-cycle-cover conjecture, \cite{Celmins1984,Preiss1981}) Any bridgeless
graph $G$ (not necessarily cubic) contains five even subgraphs such that any
edge of $G$ belongs to exactly
two of them.
\end{conjecture} It can be shown that the Petersen graph is the only 2-edge-connected cubic graph that can color all bridgeless cubic graphs \cite{Mkrt2013}. Note that \cite{Mazz11} proves that Conjecture \ref{conj:BergeFulkerson} is equivalent to proving that the edge-set of all bridgeless cubic graphs can be covered with five perfect matchings. Moreover, see the recent paper \cite{KMZ22}, which shows that every bridgeless cubic graph $G$ has a pair of perfect matchings $F_1$ and $F_2$, such that $G-(F_1\cup F_2)$ is bipartite. Note that this statement is a corollary of Conjecture \ref{conj:BergeFulkerson}, which was conjectured by Mazzuoccolo in \cite{Mazz2013}. More conjectures similar to Conjecture \ref{conj:P10conj} can be found in \cite{HakobyanAMC,Mkrt2013}. \cite{rGraphHColorings,HcoloringsRegulars} present some new results about $H$-colorings when the graphs under consideration are regular, and not necessarily are cubic.

Jaeger, in \cite{Jaeger1985}, introduced an equivalent formulation of Conjecture \ref{conj:P10conj}. Let $c$ be an edge-coloring of $G$. For a vertex $v$ of $G$, let $S_{c}(v)$ be the set of colors that edges incident to $v$ receive. If $uv$ is an edge of a cubic graph $G$, then note that $3\leq |S_{c}(u)\cup S_{c}(v)|\leq 5$.

\begin{definition}\label{def:poorrich} Suppose $ab$ is an edge of a cubic graph $G$, and $c$ is an edge-coloring of $G$. Then:
\begin{itemize}
    \item $ab$ is called poor with respect to $c$, if $|S_{c}(a)\cup S_{c}(b)|=3$,

    \item $ab$ is called abnormal with respect to $c$, if $|S_{c}(a)\cup S_{c}(b)|=4$,

    \item $ab$ is called rich with respect to $c$, if $|S_{c}(a)\cup S_{c}(b)|=5$.
\end{itemize}

\end{definition}

Edge-colorings having only poor edges are trivially $3$-edge-colorings of $G$. Also edge-colorings having only rich edges have been considered in the last years, and they are called strong edge-colorings. In \cite{Jaeger1985}, Jaeger focused on the case when all edges must be either poor or rich.

\begin{definition}\label{def:normal}
An edge-coloring $c$ of a cubic graph is normal, if $G$ does not contain abnormal edges with respect to $c$. In other words, any edge is rich or poor with respect to $c$. 
\end{definition} 

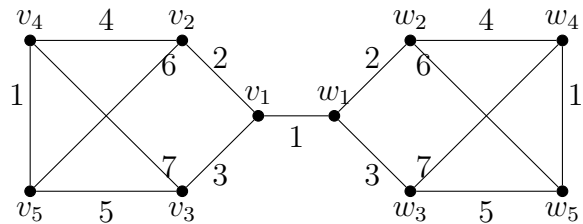
\begin{figure}[!htbp]
\begin{center}

\begin{tikzpicture}[scale=0.5]


 \node at (0, 0.55) {$v_1$};
   \node at (-2, 2.55) {$v_2$};
   \node at (-2, -2.55) {$v_3$};
   \node at (-6, 2.55) {$v_4$};
    \node at (-6, -2.55) {$v_5$};
    
    \node at (2, 0.55) {$w_1$};
   \node at (4, 2.55) {$w_2$};
   \node at (4, -2.55) {$w_3$};
   \node at (8, 2.55) {$w_4$};
    \node at (8, -2.55) {$w_5$};
    
    \node at (1, -0.55) {$1$};
    \node at (-6.35, 0.55) {$1$};
    \node at (8.35, 0.55) {$1$};
    
    \node at (-1, 1.55) {$2$};
    \node at (3, 1.55) {$2$};
    
     \node at (-1, -1.55) {$3$};
    \node at (3, -1.55) {$3$};
    
     \node at (-4, 2.55) {$4$};
    \node at (6, 2.55) {$4$};
    
    \node at (-4, -2.55) {$5$};
    \node at (6, -2.55) {$5$};
    
    \node at (-2.35, 1.35) {$6$};
    \node at (4.35, 1.35) {$6$};
    
    \node at (-2.35, -1.35) {$7$};
    \node at (4.35, -1.35) {$7$};

 \tikzstyle{every node}=[circle, draw, fill=black!,
                        inner sep=0pt, minimum width=4pt]
    \node at (0,0) (n00) {};
    
    \node at (2,0) (n20) {};


   \node at (-2,2) (nm22) {};
   \node at (4,2) (n42) {};

   
    \node at (-2,-2) (nm2m2) {};
    \node at (4,-2) (n4m2) {};

    
     \node at (-6,2) (nm62) {};
     \node at (8,2) (n82) {};
    
   
     \node at (-6,-2) (nm6m2) {};
     \node at (8,-2) (n8m2) {};

  \draw (n00)--(n20);

 \draw (n00)--(nm22);
 \draw (n00)--(nm2m2);
 
  \draw (n20)--(n42);
 \draw (n20)--(n4m2);
 
 \draw (nm22)--(nm62);
 \draw (nm2m2)--(nm62);
 \draw (nm2m2)--(nm6m2);
 \draw (nm22)--(nm6m2);
 
 \draw (n42)--(n82);
 \draw (n4m2)--(n82);
 \draw (n4m2)--(n8m2);
 \draw (n42)--(n8m2);
 
 \draw (nm62)--(nm6m2);
 \draw (n82)--(n8m2);
 

\end{tikzpicture}
\end{center}
\caption{A cubic graph that requires $7$ colors in a normal coloring. The bridge is poor. All other edges are rich. It can be shown that $\chi'_N(G)=7$.} \label{fig:exam}
\end{figure}

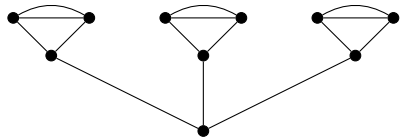
\begin{figure}[ht]
  
  \begin{center}
	
	\tikzstyle{every node}=[circle, draw, fill=black!50,
                        inner sep=0pt, minimum width=4pt]
	
		\begin{tikzpicture}
																							
			\node[circle,fill=black,draw] at (-5.5,-1) (n1) {};
			
			
			
			

			\node[circle,fill=black,draw] at (-6, -0.5) (n2) {};
																								
			\node[circle,fill=black,draw] at (-5,-0.5) (n3) {};
																								
			\node[circle,fill=black,draw] at (-3.5,-1) (n4) {};
																								
			\node[circle,fill=black,draw] at (-4, -0.5) (n5) {};
																								
			\node[circle,fill=black,draw] at (-3,-0.5) (n6) {};
																								
			\node[circle,fill=black,draw] at (-1.5,-1) (n7) {};
																								
			\node[circle,fill=black,draw] at (-2, -0.5) (n8) {};
																								
			\node[circle,fill=black,draw] at (-1,-0.5) (n9) {};
																								
			\node[circle,fill=black,draw] at (-3.5,-2) (n10) {};

			\path[every node]
			(n1) edge  (n2)

			edge  (n3)
			edge (n10)
																								   	
			(n2) edge (n3)
			edge [bend left] (n3)
																								       
			(n3) 
			(n4) edge (n5)
			edge (n6)
			edge (n10)
																								    
			(n5) edge (n6)
			edge [bend left] (n6)
			(n6)
																								   
			(n7) edge (n8)
			edge (n9)
			edge (n10)
																								    
			(n8) edge (n9)
			edge [bend left] (n9)
																								  
			;
		\end{tikzpicture}
																
	\end{center}
								
	\caption{A cubic graph that does not admit a normal $k$-edge-coloring for any $k\geq 1$.}
	\label{fig:SylvGraph}

\end{figure}

It is straightforward that an edge coloring which assigns a different color to every edge of a simple cubic graph is normal since all edges are rich. Hence, we can define the normal chromatic index of a simple cubic graph $G$, denoted by $\chi'_{N}(G)$, as the smallest $k$, for which $G$ admits a normal $k$-edge-coloring. 

 Figure \ref{fig:exam} provides an example of a normal 7-edge-coloring of a cubic graph. All edges of this graph are rich in this coloring except the unique bridge which is poor. \cite{MM2020JGT} shows that $\chi_N(G)=7$ if $G$ is the graph from Figure \ref{fig:exam}. Moreover, \cite{MM2020JGT} argues that any cubic graph having a subgraph that is isomorphic to the complete graph $K_4$ with one edge subdivided has $\chi_N(G)=7$.

Note that not all cubic graphs admit a normal $k$-edge-coloring for some $k\geq 1$. Consider a cubic multi-graph containing a triangle in which one edge is of multiplicity two (see Figure \ref{fig:SylvGraph}). It can be easily seen that such a cubic graph cannot have a normal $k$-edge-coloring, as an edge of multiplicity two is going to be abnormal in any edge-coloring.

Using the notion of normal edge-colorings, in \cite{Jaeger1985}, Jaeger has shown that:

\begin{proposition}\label{prop:JaegerNormalColor}(Jaeger, \cite{Jaeger1985}) If $G$ is a cubic graph, then $P_{10}\prec G$, if and only if $G$ admits a normal $5$-edge-coloring.
\end{proposition} This implies that Conjecture \ref{conj:P10conj} can be stated as follows:

\begin{conjecture}\label{conj:5NormalConj} For any bridgeless cubic graph $G$, $\chi'_{N}(G)\leq 5$.
\end{conjecture} In this terms, Petersen Coloring Conjecture is equivalent to saying that every bridgeless cubic graph has normal chromatic index at most 5. Observe that Conjecture \ref{conj:5NormalConj} is trivial for $3$-edge-colorable cubic graphs. This is true because in any $3$-edge-coloring $c$ of a cubic graph $G$ any edge $e$ is poor, hence $c$ is a normal edge-coloring of $G$. Thus non-$3$-edge-colorable cubic graphs are the main obstacle for Conjecture \ref{conj:5NormalConj}. Structural properties of non-3-edge-colorable bridgeless cubic graphs, sometimes called snarks, are investigated in \cite{steffen:1998,steffen:2004}. Note that Conjecture \ref{conj:5NormalConj} is verified for some non-$3$-edge-colorable bridgeless cubic graphs in \cite{LucaAKCE2020,HaggSteff2013,SedlarArXiv2023,SedlarEuJC2024}. Finally, let us note that in \cite{Samal2011} the percentage of edges of a bridgeless cubic graph, which can be made poor or rich in a 5-edge-coloring, is investigated. See for a recent paper about the percent of normal edges (that is, not abnormal edges) in 5-edge-colored cubic graphs in \cite{MMM2021}. Other recent results in this direction are obtained in \cite{MM2020DAM} and \cite{PirSerSkr}.

If we consider the larger class of simple cubic graphs, without any assumption on connectivity, some interesting questions naturally arise. Indeed, examples of simple cubic graphs with $\chi'_{N}(G) > 5$ can be constructed in this class (Figure \ref{fig:exam}). Hence it is natural to ask for a possible upper bound for this parameter. 

Let us remark that any strong edge-coloring is a normal edge-coloring. Andersen has shown in \cite{Andersen1992} that any simple cubic graph admits a strong edge-coloring with ten colors. Hence ten is also an upper-bound for the normal chromatic index. The result was improved, following the approach of Andersen, in \cite{Bilkova12}, where it is shown that any simple cubic graph admits a normal edge-coloring with nine colors. In \cite{MM2020JGT}, it is proved that if $G$ is any simple cubic graph, then $\chi'_{N}(G)\leq 7$. This result is complemented with an infinite family of simple cubic graphs in which $\chi'_{N}(G)= 7$ \cite{MM2020JGT}. Thus, the upper bound seven is (asymptotically) best-possible. 



Once the upper bound seven for all simple cubic graphs is established, one may wonder in improving it for some interesting graph classes. Of course, the first class that comes ones mind is the class of bridgeless cubic graphs. Conjecture \ref{conj:5NormalConj} predicts an upper bound five, which is difficult to prove. Thus, the following intermediate conjecture could be an excellent step in this direction.

\begin{conjecture}\label{conj:6NormalConj} (R. \v{S}\'{a}mal, \cite{Samal2016}) For any bridgeless cubic graph $G$, $\chi'_{N}(G)\leq 6$.
\end{conjecture}

In this paper, we focus on Conjecture \ref{conj:6NormalConj}. For a given perfect matching $F$ of a bridgeless cubic graph $G$, we introduce the notion of a non-conflicting no-where zero $Z_2\times Z_2$ flow of $G/\overline{F}$ (Definition \ref{def:NonConflictFlow}). We demonstrate the usefulness of this concept by showing that if a bridgeless cubic graph $G$ has such a perfect matching, then it admits a normal 6-edge-coloring (Lemma \ref{lem:NonConflictFlowNormal6coloring}). Moreover, we relate non-conflicting no-where zero $Z_2\times Z_2$ flows in cubic graphs to a recent conjecture of Thomassen about the existence of a pair of edge-disjoint perfect matchings in highly edge-connected regular graphs \cite{Thom20} (see Lemma \ref{lem:5regular5edgeconnected}). Then we obtain the main results of the paper. The first one states that claw-free bridgeless cubic graphs $G$ have a perfect matching $F$ with respect to which $G/\overline{F}$ admits a non-conflicting no-where zero $Z_2\times Z_2$ flow (see Theorem \ref{thm:clawfree}). Moreover, one can find such a perfect matching in bridgeless cubic graphs which have a 2-factor that contains at most two cycles if $G$ is not the Petersen graph (see Theorem \ref{thm:TwoCycles2factor}). In the end of the paper, we construct infinitely many 2-edge-connected cubic graphs $G$ that do not admit a non-conflicting no-where zero $Z_2\times Z_2$ flow with respect to $\overline{F}$ for any perfect matching $F$ of $G$ (see Proposition \ref{prop:PetersenNonConflicFlow} and Theorem \ref{thm:2edgeconnectedexamples}). We conclude the paper in Section \ref{sec:conclusion}, where we summarize the paper and present some questions that could be a direction of future research. Non-defined terms and concepts can be found in \cite{west:1996}.

\section{Main results}
\label{MainSection}

In this section we obtain the main results of the paper. We will need some definitions. Let $T$ be a triangle in a cubic graph $G$ such that each edge of $T$ is of multiplicity one. If $e$ is an edge of $T$, then let $f$ be the edge of $G$ that is incident to a vertex of $T$ and is not adjacent to $e$. The edges $e$ and $f$ will be called opposite. We start with the following proposition:

\begin{proposition}
    \label{prop:MinCounterExample} The minimum counter-example to Conjecture \ref{conj:6NormalConj} is a 3-edge-connected cubic graph of girth at least four.
\end{proposition}

\begin{proof} Let $G$ be a counterexample to the statement minimizing $|V(G)|$. Clearly, $G$ is connected. Let us show that it has no 2-edge-cuts. Assume that $C=\{e_1, e_2\}$ is a 2-edge-cut. Let $G_1$ and $G_2$ be the two smaller bridgeless cubic graphs arising from the two components of $G-C$ by adding one edge connecting the two degree-two vertices in the same component. We let $h_1$ and $h_2$ be the two added edges of these two graphs, respectively. Since the graphs $G_1$ and $G_2$ are smaller, we have that they admit normal 6-edge-colorings $f_1$ and $f_2$ for $j=1,2$. By renaming the colors in $G_2$, we can always assume that the colors of $h_1$ and $h_2$ are the same, moreover, the colors appearing in the ends of $e_1$ are also the same. Now, if we color $e_1$ and $e_2$ with the color of $h_1$, then we will have that $e_1$ is always, poor, moreover if at least one of $h_1$ and $h_2$ is normal, then $e_2$ will also be normal. This means that in the resulting 6-edge-coloring $f$ of $G$ will be normal, too. Thus, $G$ must be 3-connected.

Since we do not have 2-edge-cuts in $G$, we cannot have 2-cycles in $G$. Let us show that $G$ cannot contain a triangle. On the opposite assumption, assume that $G$ contain a triangle $T$. Consider the bridgeless cubic graph $G/T$. Note that it is smaller than $G$. Hence it has a normal 6edge-coloring. We can extend it to a normal 6-edge-coloring of $G$ by taking the colors of edges of $T$ that of the opposite edge. Note that this leads to a normal 6-edge-coloring of $G$ in which the edges of $T$ are poor. Thus, $G$ has girth at least four. The proof is complete.
\end{proof}

\begin{remark}
    We presented the proof of Proposition \ref{prop:MinCounterExample} for the sake of completeness. It is implicit in \cite{MM2020DAM}.
\end{remark}

\begin{remark}
    \label{rem:3edgecuts} In \cite{Jaeger1975}, Jaeger himself proved that the smallest counter-example to Conjectures \ref{conj:P10conj} and \ref{conj:5NormalConj} is a cyclically 4-edge-connected cubic graph. Unfortunately, we are not able to get rid of non-trivial 3-edge-cuts in the smallest counter-example for Conjecture \ref{conj:6NormalConj}. So if one is able to prove Conjecture \ref{conj:6NormalConj} for cyclically 4-edge-connected cubic graphs, it is not obvious how to derive the proof of the full conjecture.
\end{remark} Note that the situation is different for Conjecture \ref{conj:BergeFulkerson}. By answering a question raised back in \cite{Seymour}, M\'{a}\v{c}ajov\'{a} and Mazzuoccolo proved in \cite{MMBergeFulkersonCyclically5} that it suffices to prove Conjecture \ref{conj:BergeFulkerson} for cyclically 5-edge-connected cubic graphs. 

Note that Conjecture \ref{conj:6NormalConj} predicts an upper bound six for $\chi'_N$ in the class of all bridgeless cubic graphs. We would like to mention that seven is relatively easy to prove via flows \cite{Bilkova12,MM2020JGT}. The classical 8-flow theorem \cite{Zhang1997} implies that every bridgeless cubic graph $G$ admits a no-where zero $Z_2\times Z_2 \times Z_2$-flow. It can be easily verified that any such flow is a normal 7-edge-coloring of $G$ (see \cite{Bilkova12,MM2020JGT} for details). In \cite{MM2020JGT}, the authors managed to show that all simple, cubic graphs (not necessarily bridgeless) admit a normal 7-edge-coloring. The proof given in \cite{MM2020JGT} heavily uses flows. So an interesting question is whether flows can be helpful in proving Conjecture \ref{conj:6NormalConj}. If they are then one may wonder what kind of flow results we need in order to prove Conjecture \ref{conj:6NormalConj}. The following approach is due to Mazzuoccolo:

\begin{definition}
    \label{def:NonConflictFlow} Let $G$ be a bridgeless cubic graph, $F$ be a perfect matching of $G$ and $\overline{F}$ be the complementary 2-factor of $F$ in $G$. A no-where zero $Z_2\times Z_2$-flow $\theta$ of $G/\overline{F}$ is called non-conflicting with respect to $F$ (or with respect to $\overline{F}$), if $\overline{F}$ contains no edge $e=uv$, such that $u$ is incident to an edge with $\theta$-value $\alpha$ and $v$ is incident to an edge with $\theta$-value $\beta$. 
\end{definition} In this paper, an edge $e=uv\in E(\overline{F})$, such that $u$ is incident to an edge with $\theta$-value $\alpha$ and $v$ is incident to an edge with $\theta$-value $\beta$ will be called a conflicting edge or just a conflict. 

\begin{remark}
    \label{rem:traingle} Note that if $\overline{F}$ contains a triangle, then $G/\overline{F}$ does not admit a non-conflicting no-where zero $Z_2\times Z_2$-flow with respect to $\overline{F}$.
\end{remark}

\begin{proof} Suppose $\overline{F}$ contains a triangle $T$. Since $G$ is bridgeless, the vertex of $G/\overline{F}$ corresponding to $T$ has degree three. Thus, any no-where zero $Z_2\times Z_2$-flow $\theta$ of $G/\overline{F}$ has exactly one edge of values $\alpha$, $\beta$ and $\alpha+\beta$. Hence $T$ contains a conflict. The proof is complete.
\end{proof}

Lemma \ref{lem:NonConflictFlowNormal6coloring} proved below demonstrates the usefulness of non-conflicting flows for Conjecture \ref{conj:6NormalConj}.

\begin{lemma}
    \label{lem:NonConflictFlowNormal6coloring} Let $G$ be a bridgeless cubic graph and let $F$ be a perfect matching in it. If $G/\overline{F}$ admits a non-conflicting no-where zero $Z_2\times Z_2$ flow with respect to $\overline{F}$, then $\chi'_N(G)\leq 6$. 
\end{lemma}

\begin{proof} We follow the approach of the proof of Lemma 5.2 in \cite{HolySkoJCTB2004}. Suppose $\theta$ is given. Consider a nowhere zero $Z_2\times Z_2\times Z_2$-flow $\mu$ of $G$ obtained from $\theta$ as follows: for any edge $h\in M$, we define the triple $\mu(h)$ as follows: $\mu(h)=(0,\theta(h))$. Then let $C$ be any cycle of $\overline{M}$. Let $x_0$ be any element of $Z_2\times Z_2\times Z_2$, whose first coordinate is $1$. Assign $x_0$ to an edge of $C$. Then observe that the rest of the values of edges of $C$ are defined uniquely in $\mu$. Moreover, the first coordinate of the values of $\mu$ on $C$ is $1$. Hence for any edges $h_1\in M$ and $h_2\in \overline{M}$, we have $\mu(h_1)\neq \mu(h_2)$. Also observe that for different cycles of $\overline{M}$ we can choose $x_0$ differently. It can be easily checked that $\mu$ is a nowhere zero $Z_2\times Z_2\times Z_2$-flow of $G$. Hence, it is a normal 7-edge-coloring of $G$ as we mentioned before. Now let us consider an edge-coloring of $G$ obtained from $\mu$ by changing the values of all edges $e$ with $\mu(e)=(0, \beta)$ to $\mu(e)=(0, \alpha)$. Note that the resulting coloring is not a flow, however it is a normal 6-edge-coloring since $\theta$ was non-conflicting by assumption. The proof is complete.
\end{proof} As an approach towards Conjecture \ref{conj:6NormalConj}, in \cite{MM2020DAM} the following conjecture is presented:

\begin{conjecture}
\label{conj:NonConflictingFlow6} (\cite{MM2020DAM}) Let $G$ be a 3-edge-connected cubic graph different from the Petersen graph. Then $G$ admits a nowhere zero $Z_2\times Z_2\times Z_2$-flow $f$, such that there are two elements $x, y \in Z_2\times Z_2\times Z_2$ with
\begin{enumerate}
    \item [(1)] $f^{-1}(\{x, y \})$ is a matching in $G$,
    
    \item [(2)] there is no edge $e=uv$ of $G$, such that $u$ is incident to an edge $e_u$ and $v$ is incident to an edge $e_v$ with $f(e_u)=x$ and $f(e_v)=y$.
\end{enumerate}
\end{conjecture} \cite{MM2020DAM} shows that Conjecture \ref{conj:NonConflictingFlow6} implies Conjecture \ref{conj:6NormalConj}.

\begin{remark}\label{rem:DAMpaperConjecture4}
    Note that the no-where zero $Z_2\times Z_2\times Z_2$ flow $\mu$ that we constructed in the proof of Lemma \ref{lem:NonConflictFlowNormal6coloring} satisfies the assumption of Conjecture \ref{conj:NonConflictingFlow6}.
\end{remark}

Lemma \ref{lem:NonConflictFlowNormal6coloring} demonstrates the usefulness of non-conflicting no-where zero $Z_2\times Z_2$ flows for obtaining normal 6-edge-colorings of cubic graphs. Rather surprisingly, as our next statement demonstrates, they are useful for a recent conjecture of Thomassen \cite{Thom20} that deals with the problem of existence of pairs of edge-disjoint perfect matchings in highly connected regular graphs. In \cite{Thom20}, Thomassen conjectured that there is an integer $r_0$, such that any $r$-regular $r$-edge-connected graph on even number of vertices with $r\geq r_0$ contains a pair of edge-disjoint perfect matchings. Note that snarks (graphs like $P_{10}$) demonstrate that $r_0\geq 4$. Now, using completely different approaches Rizzi in \cite{Rizzi1999} and Mazzuoccolo in \cite{Mazz13} constructed a $4$-regular 4-edge-connected graph on even number of vertices in which every pair of perfect matchings have a common edge. The two constructions lead to the same 4-regular graph up to graph isomorphisms. This implies that $r_0\geq 5$. Now, we prove a statement that shows the usefulness of non-conflicting no-where zero $Z_2\times Z_2$ flows for this conjecture when $r=5$.

\begin{lemma}
    \label{lem:5regular5edgeconnected} Let $H$ be any 5-regular 5-edge-connected graph. Consider a bridgeless cubic $G$ obtained from $H$ by replacing every vertex of $H$ with a cycle of length five. Let $\overline{F}$ be the 2-factor of $G$ comprised of these five cycles corresponding to vertices of $H$. If $G$ admits a non-conflicting no-where zero $Z_2\times Z_2$ flow with respect to $\overline{F}$, then $H$ contains a pair of edge-disjoint perfect matchings.
\end{lemma}

\begin{proof} Suppose $\theta$ is a non-conflicting no-where zero $Z_2\times Z_2$ flow of $G$ with respect to $\overline{F}$. Since all cycles of $\overline{F}$ are odd (in fact, they are of length five), we have that there should be odd number of edges with $\theta$-value $\alpha$, $\beta$ and $\alpha+\beta$. In particular, there should be edges $e\in F$ and $f\in F$ such that $\theta(e)=\alpha$ and $\theta(f)=\beta$. If around some five cycle of $\overline{F}$ there are at least three edges with $\theta$-value $\alpha$ or at least three edges with $\theta$-value $\beta$, then since our cycles of $\overline{F}$ are of length five, it can be easily checked that there is a conflict. Thus, around every five cycle of $\overline{F}$, there are exactly one edge of $\theta$-value $\alpha$ and one edge of $\theta$-value $\beta$. Now, note that these edges induce a pair of edge-disjoint perfect matchings in $H$. The proof is complete.
\end{proof}

Let us note that it is an open problem whether there is a 5-edge-connected 5-regular graph that is not 5-edge-colorable \cite{SidmaPerfects}. Next, we discuss non-conflicting no-where zero $Z_2\times Z_2$ flows in bipartite graphs and 3-edge-colorable cubic graphs. We start with:

\begin{observation}
    \label{obs:2factorEvenCycles} Let $G$ be a bridgeless cubic graph and let $F$ be a perfect matching of $G$. If all cycles of $\overline{F}$ are even then $G/\overline{F}$ admits a non-conflicting no-where zero $Z_2\times Z_2$ flow with respect to $\overline{F}$. 
\end{observation}

\begin{proof} Let $x\in \{\alpha, \beta, \alpha+\beta\}$. For any edge $e\in E(G/\overline{F})$, set $\theta(e)=x$. Since all cycles of $\overline{F}$ are even and $x+x=0$, we have that $\theta$ is a non-conflicting no-where zero $Z_2\times Z_2$ flow with respect to $\overline{F}$. The proof is complete.
\end{proof}

\begin{observation}
    \label{obs:cubicbips} For every perfect matching $F$ of an arbitrary bipartite cubic graph $G$ $G/\overline{F}$ admits a non-conflicting no-where zero $Z_2\times Z_2$ flow with respect to $\overline{F}$. 
\end{observation}

\begin{proof} Since $G$ is bipartite, all cycles in $\overline{F}$ are even. Hence the statement follows from Observation \ref{obs:2factorEvenCycles}. The proof is complete.
\end{proof}

\begin{observation}
    \label{obs:3edgecolorablecubic} Every 3-edge-colorable cubic graph $G$ has a perfect matching $F$ such that $G/\overline{F}$ admits a non-conflicting no-where zero $Z_2\times Z_2$ flow with respect to $\overline{F}$. 
\end{observation}

\begin{proof} If $G$ is 3-edge-colorable, then there is a perfect matching $F$ such that $\overline{F}$ has only even cycles. Hence the statement follows from Observation \ref{obs:2factorEvenCycles}. The proof is complete.
\end{proof}

Recall that a graph $G$ is claw-free, if it does not contain four vertices, such that the subgraph of $G$ induced on these vertices is isomorphic to $K_{1,3}$. Next, we are going to show that claw-free bridgeless cubic graphs have a perfect matching $F$ such that $G/\overline{F}$ admits a non-conflicting no-where zero $Z_2\times Z_2$ flow with respect to $\overline{F}$. Since in \cite{MM2020DAM}, it is shown that all such graphs admit a normal 6-edge-coloring, our statement is going to strengthen the corresponding result from \cite{MM2020DAM}. Proving upper bounds for $\chi'_N$ in this class is important as if one shows that all claw-free simple bridgeless cubic graphs admit a normal 5-edge-coloring then Conjecture \ref{conj:5NormalConj} follows (see \cite{MM2020DAM} for all details on this). 

We will need some results on claw-free simple cubic graphs. In \cite{ChudSeyClawFreeChar}, arbitrary claw-free graphs are characterized. In \cite{sang-il_oum:2011}, Oum has characterized simple, claw-free bridgeless cubic graphs. In order to formulate Oum's result, we need some definitions. In a claw-free simple cubic graph $G$ any vertex belongs to one, two, or three triangles. If a vertex $v$ belongs to three triangles of $G$, then the component of $G$ containing $v$ is isomorphic to $K_4$ (Figure \ref{fig:K4}). An induced subgraph of $G$ that is isomorphic to $K_4-e$ is called a diamond \cite{sang-il_oum:2011}. It can be easily checked that in a claw-free cubic graph no two diamonds intersect.  

\begin{figure}[ht]
\centering
\begin{minipage}[b]{.5\textwidth}
  \begin{center}
\begin{tikzpicture}[scale=0.35]

 \tikzstyle{every node}=[circle, draw, fill=black!,
                        inner sep=0pt, minimum width=4pt]
  
  \node[circle,fill=black,draw] at (-2,2) (n2) {};

   \node[circle,fill=black,draw] at (-2,-2) (n3) {};

    \node[circle,fill=black,draw] at (-6,2) (n4) {};
    
    \node[circle,fill=black,draw] at (-6,-2) (n5) {};

 \draw (n2)--(n3);
 
 \draw (n2)--(n4);
 \draw (n3)--(n4);
 \draw (n3)--(n5);
 \draw (n2)--(n5);
 
 \draw (n4)--(n5);

\end{tikzpicture}
\end{center}
\caption{The graph $K_4$.} \label{fig:K4}
\end{minipage}%
\begin{minipage}[b]{.5\textwidth}

  \begin{center}
	
	\tikzstyle{every node}=[circle, draw, fill=black!50,
                        inner sep=0pt, minimum width=4pt]
	
		\begin{tikzpicture}

			

			\node[circle,fill=black,draw] at (-6, -0.5) (n2) {};
																	
			\node[circle,fill=black,draw] at (-5,-0.5) (n3) {};

			\path[every node]

			(n2) edge (n3)
			edge [bend left] (n3)
                edge [bend right] (n3)
																							  
			;
		\end{tikzpicture}
																
	\end{center}
								
	\caption{The graph $K_{2}^{3}$.}
	\label{fig:K23}
\end{minipage}
\end{figure}

A string of diamonds of $G$ is a maximal sequence $F_{1},...,F_{k}$ of diamonds, in which $F_{i}$ has a vertex adjacent to a vertex of $F_{i+1}$, $1\leq i \leq k-1$.  A string of diamonds has exactly two vertices of degree two, which are called the head and the tail of the string. Replacing an edge $e = uv$ with a string of diamonds with the head $x$ and the tail $y$ is to remove $e$ and add edges $(u,x)$ and $(v,y)$.

If $G$ is a connected claw-free simple cubic graph such that each vertex lies in a diamond, then $G$ is called a ring of diamonds. It can be easily checked that each vertex of a ring of diamonds lies in exactly one diamond. As in \cite{sang-il_oum:2011}, we require that a ring of diamonds contains at least two diamonds.

\begin{proposition}\label{prop:OumClawfreebridgelessCharac} (Oum, \cite{sang-il_oum:2011}) $G$ is a connected claw-free simple bridgeless cubic graph, if and only if
\begin{enumerate}
    \item [(1)] $G$ is isomorphic to $K_4$, or
    
    \item [(2)] $G$ is a ring of diamonds, or
    
    \item [(3)] there is a connected bridgeless cubic graph $H$, such that $G$ can be obtained from $H$ by replacing some edges of $H$ with strings of diamonds, and by replacing any vertex of $H$ with a triangle.
\end{enumerate}
\end{proposition} We would like to present the following simple extension of Proposition \ref{prop:OumClawfreebridgelessCharac} when $G$ may have a parallel edge or just a cycle of length two.

\begin{proposition}\label{prop:AnushVahanClawfreebridgelessCharac} (\cite{HakobyanAUJC}) $G$ is a connected claw-free bridgeless cubic graph, if and only if
\begin{enumerate}
    \item [(1)] $G$ is isomorphic to $K_4$ (Figure \ref{fig:K4}) or to $K_2^3$ (Figure \ref{fig:K23}), or
    
    \item [(2)] $G$ is a ring of diamonds or 2-cycles, or
    
    \item [(3)] there is a connected bridgeless cubic graph $H$, such that $G$ can be obtained from $H$ by replacing some edges of $H$ with strings of diamonds or 2-cycles (Figure \ref{fig:StringDiamonds2Cycles}), and by replacing any vertex of $H$ with a triangle.
\end{enumerate}
\end{proposition} 

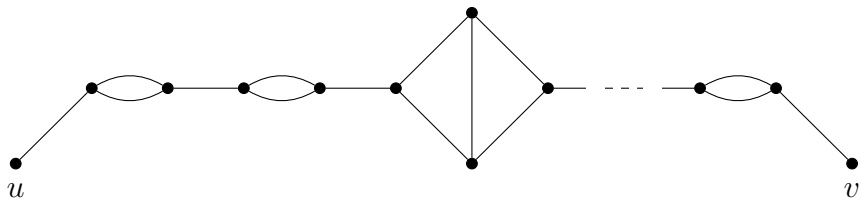
\begin{figure}[ht]
  
  \begin{center}

		\begin{tikzpicture}

                \node at (-1, -1.35) (u) {$u$};
                \node at (10, -1.35) (u) {$v$};
                \draw[dashed] (6.75,0)--(7.25,0);

                 \tikzstyle{every node}=[circle, draw, fill=black!50,
                        inner sep=0pt, minimum width=4pt] 
                        
			\node[circle,fill=black,draw] at (0,0) (n00) {};
			
			
			
			

			\node[circle,fill=black,draw] at (-1,-1) (nm1m1) {};

                \node[circle,fill=black,draw] at (1,0) (n10) {};
                \node[circle,fill=black,draw] at (2,0) (n20) {};
                \node[circle,fill=black,draw] at (3,0) (n30) {};

                \node[circle,fill=black,draw] at (4,0) (n40) {};

                \node[circle,fill=black,draw] at (6,0) (n60) {};
                \node[circle,fill=black,draw] at (5,1) (n51) {};
                \node[circle,fill=black,draw] at (5,-1) (n5m1) {};

                \node[circle,fill=black,draw] at (8,0) (n80) {};
                \node[circle,fill=black,draw] at (9,0) (n90) {};

                 \node[circle,fill=black,draw] at (10,-1) (n10m1) {};

			\path[every node]
			(n00) edge  (nm1m1)
                
                (n00) edge [bend left] (n10)
                edge [bend right] (n10)

                (n10) edge (n20)

                (n20) edge [bend left] (n30)
                edge [bend right] (n30)

                (n30) edge (n40)

                (n40) edge (n51)
                (n40) edge (n5m1)
                (n60) edge (n51)
                (n60) edge (n5m1)
                (n51) edge (n5m1)

                (n60) edge (6.5,0)
                (7.5,0) edge (n80)

                (n80) edge [bend left] (n90)
                edge [bend right] (n90)

                (n90) edge (n10m1)

			;
		\end{tikzpicture}
																
	\end{center}
								
	\caption{Replacing an edge $uv$ with a string of diamonds or 2-cycles.}
	\label{fig:StringDiamonds2Cycles}

\end{figure}

It is known that:
\begin{theorem}
    \label{thm:Sch} (\cite{Lovasz,Sch}) Every edge of a bridgeless cubic graph $G$ belongs to some perfect matching of $G$.
\end{theorem}

Using the theory of fractional perfect matchings and Edmonds' matching polytope theorem, it is not hard to prove: 

\begin{theorem}
    \label{thm:PrescribedEdge3cuts} (\cite{KKN2005,KSSidma2008,Zhang1997}) For every edge $e$ of a bridgeless cubic graph $G$ there is a perfect matching $F$ of $G$ such that $e\in F$ and $F$ intersects all 3-edge-cuts in a single edge.
\end{theorem}

We are ready to prove our main result for claw-free bridgeless cubic graphs.

\begin{theorem}
    \label{thm:clawfree} For every claw-free bridgeless cubic graph $G$ and its edge $e$ $G$ has a perfect matching $F$ such that $e\in F$ and $G/\overline{F}$ admits a non-conflicting no-where zero $Z_2\times Z_2$ flow with respect to $\overline{F}$. 
\end{theorem} 

\begin{proof} Take a perfect matching $F$ of $G$, such that $e\in F$ and $F$ intersects all 3-edge-cuts of $G$ in a single edge (see Theorem \ref{thm:PrescribedEdge3cuts}). Note that $F$ intersects all triangles of $G$ in a single edge. Moreover, note that $G/\overline{F}$ is bridgeless and contains no 3-edge-cuts. Hence, by \cite{Jaeger1979}, $G/\overline{F}$ admits a no-where zero $Z_2\times Z_2$ flow. Consider all no-where zero $Z_2\times Z_2$ flows of $G/\overline{F}$ and among them choose the one with smallest number of conflicts. Let us show that this number is zero.

Suppose there is a conflict with edges $f\in F$ and $f'\in F$. Then at least one of $f$ and $f'$ belongs to a new triangle of $G$, that is, this triangle was replacing a vertex of $H$. Thus, at least one of $f$ and $f'$ is a chord in the corresponding 2-factor $\overline{F}$. Hence, it is a loop in $G/\overline{F}$. Thus, by changing the flow value there to $\alpha+\beta$, we will still obtain a no-where zero $Z_2\times Z_2$ flow with less conflicts. This contradicts our choice. Thus, $F$ contains the edge $e$ and $G/\overline{F}$ admits a no-where zero $Z_2\times Z_2$ flow with respect to $F$ as needed. The proof is complete.
\end{proof}

Combined with Lemma \ref{lem:NonConflictFlowNormal6coloring} we get the following result from \cite{MM2020DAM}:
\begin{corollary}
    All claw-free bridgeless cubic graphs admit a normal 6-edge-coloring.
\end{corollary}

In \cite{MM2020DAM}, it is shown that every cubic permutation graph admits a normal 6-edge-coloring. Our next result strengthens this statement by showing that all bridgeless cubic graph containing a 2-factor having at most two cycles admit such an edge-coloring.

\begin{theorem}
    \label{thm:TwoCycles2factor} Let $G$ be a bridgeless cubic graph containing a 2-factor having at most two cycles. Then $G/\overline{F}$ admits a non-conflicting no-where zero $Z_2\times Z_2$ flow with respect to some 2-factor $\overline{F}$ unless $G$ is the Petersen graph.
\end{theorem} 

\begin{proof} Our proof is by induction on the number of vertices of $G$. Clearly, our statement is true when $|V|=2$. By induction assume that the statement is true for all graphs with less than $|V|$ vertices, and let us consider a bridgeless cubic graph $G$ containing a 2-factor $\overline{F}$ having at most two cycles. We will use some ideas given in \cite{MM2020DAM} for cubic permutation graphs. Note that if the number of cycles in $\overline{F}$ is one or it is two and both of the cycles are even, then our statement follows from Observation \ref{obs:2factorEvenCycles}. Thus, we can focus on the case when $\overline{F}$ has two odd cycles $C_1$ and $C_2$. Note that in great contrast with permutation cubic graphs our cycles $C_1$ and $C_2$ may have chords. Let $n$ be the number of edges of $F$ joining $C_1$ to $C_2$. Moreover, let $u_1,...,u_n$ be these $n$ vertices of $C_1$ that are joined to vertices $v_1,...,v_n$ from $C_2$, respectively. Since $C_1$ and $C_2$ are odd cycles, we have that $n$ is odd.

Let $n_1$ and $n_2$ be the number of pairs from $u_1,...,u_n$ and $v_1,...,v_n$, that are adjacent on $C_1$ and $C_2$, respectively. Note that we count the number of pairs on cycles plus chords. We have $n_1, n_2\leq n$. We consider two cases.

\medskip

Case 1: $n\geq 5$. Hence
\begin{equation}\label{eq:PairsNonPairs}
    \binom{n}{2}-n_1\geq \binom{n}{2}-n\geq n \geq n_2,
\end{equation} as $n\geq 5$. Suppose that $\binom{n}{2}-n_1>n_2$. Then, since the expression in the left side of (\ref{eq:PairsNonPairs}) counts the number of non-pairs in $C_1$, there is a non-edge $u_i, u_j$ on $C_1$ such that $v_i, v_j$ is a non-edge on $C_2$. Define a function $\theta$ as follows: $\theta(u_iv_i)=\alpha$, $\theta(u_{j}v_{j})=\beta$, and on remaining edges of $F$, we set the value of $\theta$ as $\alpha+\beta$. Note that since $C_1$ and $C_2$ are odd cycles, there are exactly one edge of $\theta$-value $\alpha$ and $\beta$, and they do not form a conflict, we have that $\theta$ is a non-conflicting no-where zero $Z_2\times Z_2$ flow with respect to the 2-factor $\overline{F}$.

Thus, it remains to consider the case $\binom{n}{2}-n_1=n_2$. From the chain of inequalities (\ref{eq:PairsNonPairs}) we have that
\[\binom{n}{2}-n_1= \binom{n}{2}-n= n = n_2.\]
This implies that $n=5$ and $C_1$, $C_2$ have no chords. Thus, either $G$ is 3-edge-colorable hence we have the statement via Observation \ref{obs:3edgecolorablecubic} or $G$ is the Petersen graph.

\medskip

Case 2: $n=3$. If $\binom{n}{2}-n_1>n_2$, the same reasoning from Case 1 works. Thus, we can assume $\binom{n}{2}-n_1\leq n_2$, or
\[3=\binom{3}{2}=\binom{n}{2}\leq n_1+n_2.\]
Since $n_1, n_2\leq n=3$, this means that either $n_1\geq 1$ and $n_2\geq 1$ or $(n_1, n_2)\in \{(3,0), (0,3)\}$. 

\medskip

Case 2a: $n_1\geq 1$ and $n_2\geq 1$. Consider the graphs $G/V(C_1)$ and $G/V(C_2)$. Note that since $n_1\geq 1$ the new vertex of $G/V(C_1)$ is adjacent to two consecutive vertices of $C_1$. Similarly, note that since $n_2\geq 1$, the new vertex of $G/V(C_2)$ is adjacent to two consecutive vertices of $C_2$. Note that these observations imply that both of $G/V(C_1)$ and $G/V(C_2)$ are Hamiltonian, as it is not hard to extend $C_1$ to a Hamiltonian cycle of $G/V(C_1)$, and extend $C_2$ to a Hamiltonian cycle of $G/V(C_2)$. Thus, both of these graphs are 3-edge-colorable. Thus, $G$ is 3-edge-colorable. Hence, our statement for this case follows from Observation \ref{obs:3edgecolorablecubic}.

\medskip

Case 2b: $(n_1, n_2)\in \{(3,0), (0,3)\}$. Since the cases are symmetric it suffices to consider the case $n_1=3$ and $n_2=0$. Since $n_1=3$, it means one of the cycles of the 2-factor is a triangle. Let $C_1=T=u_1u_2u_3$ be this triangle. For $i=1,2,3$ let $v_i$ be the unique neighbor of $u_i$ on $C_2$. Consider a cubic graph $H$ obtained from $G$ by removing all vertices of $T$ and $v_2$, and adding the edges $g=v_1v_3$ and $f$ connecting the two neighbors of $v_2$ on $C_2$. Note that this cubic graph is Hamiltonian as $D=C_2-v_2+f$ is a Hamiltonian cycle in it. Consider a perfect matching $F$ of $D$ that does not contain the edge $f$. Note that it does not contain the edge $g$, too. Observe that $D-F$ is 2-edge-colorable, hence all cycles of $D-F$ are even. Consider a 2-factor $\overline{F'}$ of $G$ obtained from $\overline{F}$ by replacing the edge $f$ with the two edges of $C_2$ incident to $v_2$, and replacing $g$ with four edges $v_1u_1$, $u_1u_2$, $u_2u_3$ and $u_3v_3$. Let $C_f$ and $C_g$ be the cycles of $\overline{F'}$ such that $C_f$ contains the edges incident to $v_2$, and $C_g$ be the cycle containing the edge $v_1u_1$. Note that if $C_f=C_g$ or these cycles are different but both of them are of even length, then by defining $\theta(e)=\alpha+\beta$ for every edge $e\in F'$, we will obtain a non-conflicting no-where zero $Z_2\times Z_2$ flow of $G/\overline{F'}$ with respect to the 2-factor $\overline{F'}$. Thus it remains to consider the case that $C_f$ and $C_g$ are different and both of them are of odd length. Note that $C_f$ and $C_g$ are the only odd cycles of $\overline{F'}$. Hence, if we set $\theta(e)=\alpha+\beta$ for every edge $e\in F'$, we will have that the definition of no-where zero $Z_2\times Z_2$ flow of $G/\overline{F'}$ is violated only at the two vertices of $G/\overline{F'}$ corresponding to $C_f$ and $C_g$.

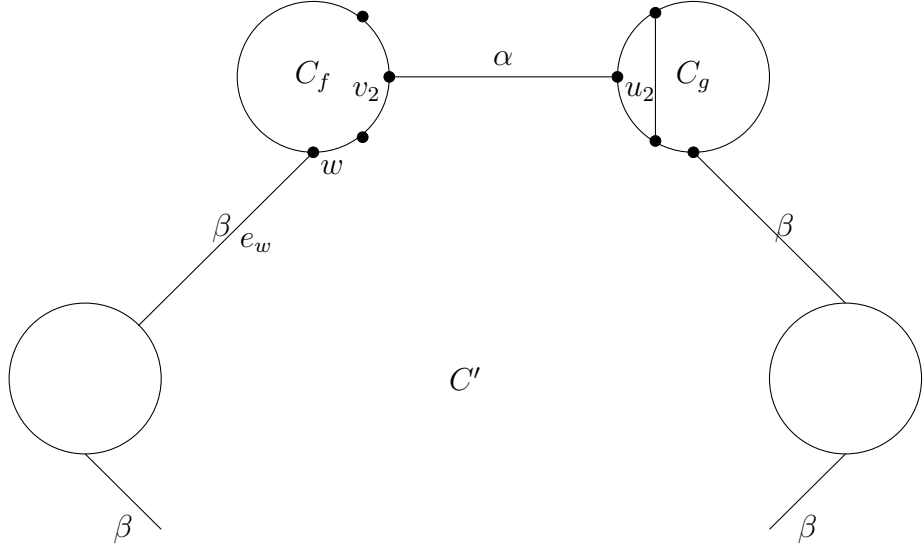
\begin{figure}[ht]
  
  \begin{center}

		\begin{tikzpicture}


                \node at (-0.75, -1.2) {$w$};
                 \node at (-1.75, -2.2) {$e_w$};

                 \node at (-0.3, -0.2) {$v_2$};
                 \node at (3.3, -0.2) {$u_2$};

                \node at (1, -4) {$C'$};
                 \node at (-1, 0) {$C_f$};
                 \node at (4, 0) {$C_g$};

                 \node at (1.5, 0.25) {$\alpha$};
                 \node at (-2.2, -2) {$\beta$};
                 \node at (5.2, -2) {$\beta$};

                 \node at (-3.5, -6) {$\beta$};
                 \node at (5.5, -6) {$\beta$};

                \tikzstyle{every node}=[circle, draw, fill=black!50,inner sep=0pt, minimum width=4pt]

                \draw (-1,0) circle (1cm); 

                \draw (4,0) circle (1cm); 

                \draw (-4,-4) circle (1cm);
                \draw (6,-4) circle (1cm);

                \draw (-1,-1) -- (-3.3,-3.3);
                \draw (4, -1) -- (6, -3);

                \draw (-4, -5) -- (-3, -6);
                \draw (6, -5) -- (5, -6);
																							
			\node[circle,fill=black,draw] at (0,0) (v2) {};

                \node[circle,fill=black,draw] at (3, 0) (u2) {};

                \node[circle,fill=black,draw] at (-1,-1) (w) {};
                 \node[circle,fill=black,draw] at (4,-1)  {};

                \node[circle,fill=black,draw] at (-0.35, 0.8) {};
                \node[circle,fill=black,draw] at (-0.35, -0.8) {};

                \node[circle,fill=black,draw] at (3.5, 0.85) (x) {};
                \node[circle,fill=black,draw] at (3.5, -0.85) (y) {};

			
			
			
			


			\path[every node]
			(v2) edge  (u2)
                (x) edge (y)

			;
		\end{tikzpicture}
																
	\end{center}
								
	\caption{The simple cycle $C'$ of $G/\overline{F'}$ containing the edges $u_2v_2$ and $e_w$.}
	\label{fig:SympleCycleCprime}

\end{figure}

Suppose that $C_f$ has a vertex $w$ not adjacent to $v_2$ such that the edge $e_w$ of $F'$ incident to $w$ connects $w$ to a vertex outside $V(C_f)$ (Figure \ref{fig:SympleCycleCprime}). Since $G$ is bridgeless, we have that $G/\overline{F'}$ is bridgeless, too. Hence $G/\overline{F'}$ contains a simple cycle $C'$ containing $u_2v_2$ and $e_w$ (Figure \ref{fig:SympleCycleCprime}). We modify the above defined function $\theta$ as follows: it is equal to $\theta(e)=\alpha+\beta$ for every edge $e\in F'$ except on edges of $C'$, where we have $\theta(u_2v_2)=\alpha$ and $\theta(e)=\beta$ for all edges of $e\in C'-\{u_2v_2\}$ (Figure \ref{fig:SympleCycleCprime}). Note that $\theta$ is a no-where zero $Z_2\times Z_2$ flow of $G/\overline{F'}$. Let us show that it is non-conflicting. Note that we have only one edge with $\theta$-value $\alpha$, which is $u_2v_2$. Note that the edge $u_1u_3$ is chord of $\overline{F'}$ hence we can take its flow value $\alpha+\beta$, hence we do not violate the definition of the flow. On the other hand, the two edges of $F'$ that are incident to a vertex adjacent to $v_2$ do not lie on $C'$, the flow values of these edges will be $\alpha+\beta$, too. Hence we will not have a conflict.

Thus, we are left with the case when $V(C_f)$ is joined to $V\backslash V(C_f)$ with exactly three edges. One of them is $u_2v_2\in F'$ and the others are $g\in F'$ and $h\in F'$. Moreover, $g$ and $h$ are incident to vertices that are adjacent to $v_2$ on $C_2$. Let $K=\{u_2v_2, g, h\}$ be the non-trivial 3-edge-cut joining $V(C_f)$ to $V\backslash V(C_f)$. Consider the graphs 
\[H_1=G/V(C_f) \text{ and } H_2=G/(V\backslash V(C_f)).\]
Note that since the new vertex of $H_2$ corresponding to $V\backslash V(C_f)$ is joined to two neighboring vertices on $C_f$, we have that $H_2$ has a Hamiltonian cycle, hence it is 3-edge-colorable. On the other hand, note that $H_1$ is a bridgeless cubic graph containing a 2-factor with at most two cycles ($C_1$ and the restriction of $C_2$ to $H_1$). Since $|V(H_1)|<|V|$, $H_1$ contains a triangle hence it is different from $P_{10}$, by induction we have that $H_1$ has a perfect matching $F_0$, such that $H_1/\overline{F_0}$ admits a non-conflicting no-where zero $Z_2\times Z_2$ flow. By Remark \ref{rem:traingle}, $F_0$ intersects $\{u_1v_1, u_2v_2, u_3v_3\}$ in a single edge. Moreover, trivially $F_0$ intersects $K$ in a single edge. Let this edge of $K$ be $t$. Suppose that the value of non-conflicting no-where zero $Z_2\times Z_2$ flow of $H_1/\overline{F_0}$ on $t$ is $x$. Since $H_2$ is 3-edge-colorable, we have that $H_2$ has a perfect matching $J$, such that $t\in J$ and $H_2-J$ is comprised of even cycles. Extend the perfect matching $F_0$ to a perfect matching $J'$ of our original graph $G$ by adding the edges of $J$ to it and taking the flow value on them as $x$. Note that we will get a perfect matching $J'$ of $G$ with respect to which $G/\overline{J}$ admits a non-conflicting no-where zero $Z_2\times Z_2$ flow. The proof is complete.
\end{proof}

Combined with Lemma \ref{lem:NonConflictFlowNormal6coloring} we get the following result from \cite{MM2020DAM}:
\begin{corollary}
    All cubic permutation graphs admit a normal 6-edge-coloring.
\end{corollary}

In the previous statements we explicitly required that the cubic graph under consideration differs from $P_{10}$. The reader probably guessed that the main reason why we did that is that it does not admit a non-conflicting flow with respect to any 2-factor $\overline{F}$. Note that since $P_{10}$ is triangle-free, presence of a triangle in a 2-factor (see Remark \ref{rem:traingle}) is not the only obstruction for the existence of a non-conflicting no-where zero $Z_2\times Z_2$ flow.
    
\begin{proposition}
\label{prop:PetersenNonConflicFlow} The Petersen graph $P_{10}$ does not admit a non-conflicting no-where zero $Z_2\times Z_2$ flow with respect to any 2-factor $\overline{F}$.
\end{proposition}

\begin{proof} The complementary 1-factor $F$ is a 5-edge-cut that separates the two 5-cycles of $\overline{F}$. Since it is an odd edge-cut, in any no-where zero $Z_2\times Z_2$ flow of $G/\overline{F}$, an odd number of $\alpha$, $\beta$ and $\alpha+\beta$ edges must appear on it. Assume that on one of them the value $\alpha$ appears. Then this edge is joined to the other four edges of $F$ with an edge of from $\overline{F}$. Thus, we cannot put the value $\beta$ on either of them. The proof is complete.
\end{proof}

The Petersen graph $P_{10}$ is exceptional in many cases. In other words, there are statements where the only counter-example to them is $P_{10}$ (Theorem \ref{thm:TwoCycles2factor}, or the main result of \cite{Lov1987} are typical cases of this phenomenon). So one may wonder whether $P_{10}$ is the only 2-edge-connected cubic graph which does not admit a non-conflicting no-where zero $Z_2\times Z_2$ flow with respect to any 2-factor $\overline{F}$. Unfortunately, this statement fails as our next result shows.

\begin{theorem}\label{thm:2edgeconnectedexamples} There exist infinitely many 2-edge-connected cubic graphs $G$ that do not admit a non-conflicting no-where zero $Z_2\times Z_2$ flow with respect to $\overline{F}$ for any perfect matching $F$ of $G$. 
\end{theorem}

\begin{proof} For $\ell\geq 1$ take $3\ell$ vertex-disjoint copies of $P_{10}-e$. Let $H_1,...,H_{3\ell}$ be these graphs. Join $H_1,...,H_{3\ell}$ cyclically by paths of length two. Now introduce $\ell$ new vertices $u_1,...,u_{\ell}$ and join each of them to exactly three central vertices of these paths of length two so that the resulting graph $G$ is cubic (see Figure \ref{fig:Large2edgeconnectedExamples}). Note that $G$ is 2-edge-connected.

\begin{figure}[ht]
  
  \begin{center}

		\begin{tikzpicture}







                 \node at (1.5,-2.4)  {$u_1$};
                 
                 \node at (-1.1, 0.5) {$P_{10}-e$}; 

                 \node at (4.1, 0.5) {$P_{10}-e$}; 

                 \node at (-4, -4) {$P_{10}-e$}; 

                 \node at (6, -4) {$P_{10}-e$}; 

                \tikzstyle{every node}=[circle, draw, fill=black!50,inner sep=0pt, minimum width=4pt]

                \draw (-1,0) circle (1cm); 

                \draw (4,0) circle (1cm); 

                \draw (-4,-4) circle (1cm);
                \draw (6,-4) circle (1cm);



                \draw (-4,-4.5) -- (-3, -5.5);

                \draw (6.5,-4.5) --(5.5,-5.5);
																							
			\node[circle,fill=black,draw] at (-0.5,0) (a1) {};
                \node[circle,fill=black,draw] at (-1,-0.5) (a2) {};

                \node[circle,fill=black,draw] at (3.5,0) (b1) {};
                \node[circle,fill=black,draw] at (4,-0.5) (b2) {};

                \node[circle,fill=black,draw] at (1.5,0) (ab) {};

                \node[circle,fill=black,draw] at (-3.5,-3.5) (c1) {};
                \node[circle,fill=black,draw] at (-4,-4.5) (c2) {};

                \node[circle,fill=black,draw] at (-2.15, -2) (ac) {};

                \node[circle,fill=black,draw] at (6,-3.5) (d1) {};
                \node[circle,fill=black,draw] at (6.5,-4.5) (d2) {};

                 \node[circle,fill=black,draw] at (5.15, -2) (bd) {};

               \node[circle,fill=black,draw] at (1.5,-2) (u) {};

			\path[every node]
                (ab) edge (a1)
                (ab) edge (b1)

                (ac) edge (a2)
                (ac) edge (c1)

                (bd) edge (b2)
                (bd) edge (d1)

                (u) edge (ab)
                (u) edge (ac)
                (u) edge (bd)

			;
		\end{tikzpicture}
																
	\end{center}
								
	\caption{Examples of arbitrary large 2-edge-connected cubic graphs.}
	\label{fig:Large2edgeconnectedExamples}

\end{figure}
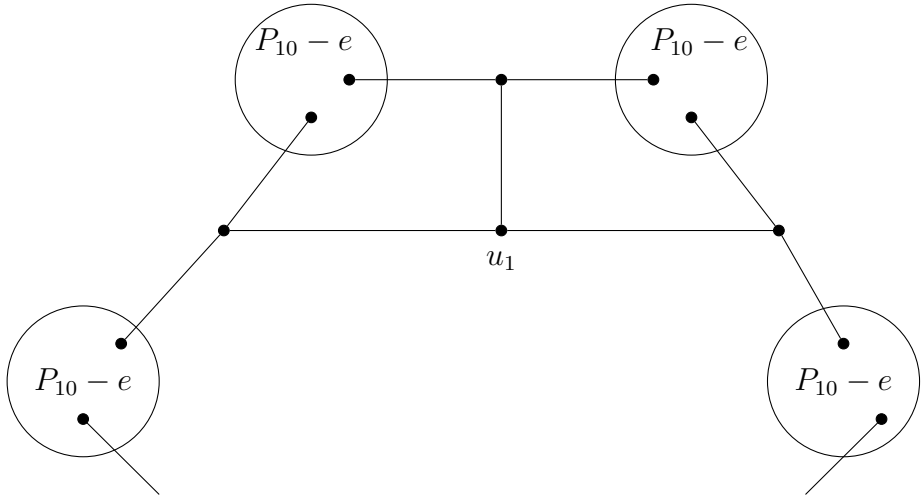
Let us  show that $G$ does not admit a non-conflicting no-where zero $Z_2\times Z_2$ flow with respect to $\overline{F}$ for any perfect matching $F$ of $G$. Let $F$ be a perfect matching of $G$, and let $\theta$ be a no-where zero $Z_2\times Z_2$ flow of $G/\overline{F}$. Consider the vertex $u_1$. Then, there is an edge $u_1w_1$ of $G$ such that $u_1w_1\notin F$. Here $w_1$ is a central vertex of a path of length two. Thus, there is $j$, $1\leq j\leq 3\ell$, such that the two edges of $G$ that connect $H_j$ to $V(G)\backslash V(H_j)$ belong to $F$. Hence, they have the same flow value with respect to $\theta$. Now, with the same approach as we did in Proposition \ref{prop:PetersenNonConflicFlow}, one can show that $H_j$ contains a conflict with respect to the flow $\theta$. The proof is complete.
\end{proof}

\section{Conclusion and future work}
\label{sec:conclusion}

In this paper, we considered the problem of existence of a perfect matching $F$ in a bridgeless cubic graph $G$ with respect to which $G/\overline{F}$ admits a non-conflicting no-where zero $Z_2\times Z_2$ flow. We demonstrated the usefulness of this concept by showing that if a bridgeless cubic graph $G$ has such a perfect matching, then it admits a normal 6-edge-coloring. Moreover, we have related non-conflicting no-where zero $Z_2\times Z_2$ flows in cubic graphs to a recent conjecture of Thomassen about the existence of a pair of edge-disjoint perfect matchings in highly edge-connected regular graphs. Our main results state that claw-free bridgeless cubic graphs $G$ have a perfect matching $F$ with respect to which $G/\overline{F}$ admits a non-conflicting no-where zero $Z_2\times Z_2$ flow. Moreover, one can find such a perfect matching in bridgeless cubic graphs which have a 2-factor that contains at most two cycles if $G$ is not the Petersen graph. In the second half of the paper, we constructed infinitely many 2-edge-connected cubic graphs $G$ that do not admit a non-conflicting no-where zero $Z_2\times Z_2$ flow with respect to $\overline{F}$ for any perfect matching $F$ of $G$.

From our perspective there are some questions that deserve further consideration. In Proposition \ref{prop:MinCounterExample}, we showed that the smallest counter-example to Conjecture \ref{conj:6NormalConj} is 3-edge-connected. Hence they do not contain 2-edge-cuts. This prompted Mazzuoccolo to ask whether $P_{10}$ is the only 3-edge-connected cubic graph which does not admit a non-conflicting no-where zero $Z_2\times Z_2$ flow with respect to $\overline{F}$ for any perfect matching $F$ of $G$? Note that since $P_{10}$ admits a normal 5-edge-coloring, it is not a counter-example to Conjecture \ref{conj:6NormalConj}. Thus, Mazzuoccolo's statement combined with Proposition \ref{prop:MinCounterExample} and Lemma \ref{lem:NonConflictFlowNormal6coloring} implies Conjecture \ref{conj:6NormalConj}. 

The author suspects that this statement is not true. Thus, Mazzuoccolo asked whether $P_{10}$ is the only cyclically 4-edge-connected cubic graph $G$ which does not have a non-conflicting no-where zero $Z_2\times Z_2$ flow with respect to $\overline{F}$ for any perfect matching $F$ of $G$. Note that if this statement is true, then cyclically 4-edge-connected cubic graphs will admit a normal 6-edge-coloring. Unfortunately, as we stated in Remark \ref{rem:3edgecuts}, this does not directly imply Conjecture \ref{conj:6NormalConj}. Nevertheless, the author thinks that this is a question deserving further consideration. The author suspects that the answer to this question should be negative, too. As a result, he would like to offer: 
\begin{conjecture}
    \label{conj:Cyclically6edgeconnectedExample} There exists infinitely many cyclically 6-edge-connected cubic graphs $G$ which do not have a non-conflicting no-where zero $Z_2\times Z_2$ flow with respect to $\overline{F}$ for any perfect matching $F$ of $G$.
\end{conjecture}

Note that in Conjecture \ref{conj:Cyclically6edgeconnectedExample} for cyclic edge-connectivity we write six and not a larger number, because of the standard Jaeger-Swart conjecture \cite{JaegerSwart1980} in the area which predicts that all cyclically 7-edge-connected cubic graphs are 3-edge-colorable. Note that Thomassen conjectures even more that all cyclically 8-edge-connected cubic graphs are Hamiltonian.

Another interesting problem by Mazzuoccolo is the following one:

\begin{problem} \label{prob:GirthAnyGexamples} For every $g\geq 3$ construct infinitely many 3-edge-connected cubic graphs $G$ of girth at least $g$ that do not admit a non-conflicting no-where zero $Z_2\times Z_2$ flow with respect to $\overline{F}$ for any perfect matching $F$ of $G$. 
\end{problem} This problem is the reminiscent of the girth conjecture in the area that was predicting that all snarks should have bounded girth. Recall that girth conjecture has been refuted by Kochol in \cite{Kochol96} where he constructed snarks of arbitrary large girth. Kochol's approach has been simplified in papers \cite{Mac22} and \cite{MacSko21}.

As we mentioned in Observation \ref{obs:cubicbips}, every bipartite cubic graph admits a non-conflicting flow with respect to every 2-factor $\overline{F}$. One can state a conjecture that predicts the converse: if a cubic graph $G$ admits a non-conflicting flow with respect to every 2-factor $\overline{F}$, then $G$ is bipartite. Unfortunately, this statement is not true. Consider a cubic graph such that its every 2-factor is a Hamiltonian cycle (such as $K_4$, Figure \ref{fig:K4}). Then it is an obvious counter-example to our conjecture by Observation \ref{obs:2factorEvenCycles}. One can construct more connected examples as follows: consider a graph $G$ which is a ring of diamonds or 2-cycles (see Theorem \ref{prop:AnushVahanClawfreebridgelessCharac}). Note that the components of every 2-factor $\overline{F}$ of such a graph are even cycles. Hence, they admit a non-conflicting no-where zero $Z_2\times Z_2$ flow with respect to $\overline{F}$ for any perfect matching $F$ of $G$ by Observation \ref{obs:2factorEvenCycles}. Note that these graphs are not bipartite as diamonds contain triangles. The examples described above are 3-edge-colorable. If one takes the graph $K_{2}^{3}$ (Figure \ref{fig:K23}) and replaces both of its vertices with a $P_{10}-v$, then the resulting graph is not 3-edge-colorable, however $G/\overline{F}$ admits a non-conflicting no-where zero $Z_2\times Z_2$-flow with respect to every 2-factor $\overline{F}$. Thus, as an interesting problem we would like to offer:

\begin{problem}
    \label{prob:CharacterizeConnectedGraphs} Characterize 2-edge-connected cubic graphs in which $G/\overline{F}$ admits a non-conflicting no-where zero $Z_2\times Z_2$-flow with respect to every 2-factor $\overline{F}$. 
\end{problem} Note that as we mentioned in Remark \ref{rem:traingle}, if $\overline{F}$ contains a triangle, then $G/\overline{F}$ does not admit a non-conflicting no-where zero $Z_2\times Z_2$-flow with respect to $\overline{F}$. Thus, graphs satisfying the conditions of Problem \ref{prob:CharacterizeConnectedGraphs} should not contain a triangle in any of its 2-factor. As rings of diamonds demonstrate, this does not necessarily mean that $G$ should not contain a triangle.

Our last conjecture states:
\begin{conjecture}
    \label{conj:HaniltonianPath} Let $G$ be a bridgeless cubic graph containing a Hamiltonian path. Then $G/\overline{F}$ admits a non-conflicting no-where zero $Z_2\times Z_2$ flow with respect to some 2-factor $\overline{F}$ unless $G$ is the Petersen graph.
\end{conjecture} Note that Conjecture \ref{conj:HaniltonianPath} implies that if a bridgeless cubic graph $G$ has a vertex $z$ such that $G-z$ has a Hamiltonian cycle, then $G/\overline{F}$ admits a non-conflicting no-where zero $Z_2\times Z_2$ flow with respect to some 2-factor $\overline{F}$ unless $G$ is the Petersen graph. The latter implies that if $G$ is a cubic graph such that for every vertex $z$ $G-z$ contains a Hamiltonian cycle, then $G/\overline{F}$ admits a non-conflicting no-where zero $Z_2\times Z_2$ flow with respect to some 2-factor $\overline{F}$ unless $G$ is the Petersen graph. Non-3-edge-colorable bridgeless cubic graphs in which  for every vertex $z$ $G-z$ contains a Hamiltonian cycle are called Hypo-Hamiltonian snarks and examples of such graphs can be found in \cite{MacSko2007} and Theorem 2.5 from \cite{SteffenHypoHamiltonian}.

\section*{Acknowledgement} The author would like to thank Giuseppe Mazzuoccolo for useful discussions over the topic of the present paper.



\bibliographystyle{elsarticle-num}



\begin{thebibliography}{99} 

\bibitem{Andersen1992} L. D. Andersen, The strong chromatic index of a cubic graph is at most 10, Discrete Mathematics, 108 (1992), 231--252.

\bibitem{Bilkova12} H. B\'{i}lkov\'{a}, Petersenovsk\'{e} obarven\'{i} a jeho varianty, Bachelor thesis, Charles University in Prague, Prague, 2012, (in Czech).

\bibitem{Celmins1984} A. U. Celmins, On cubic graphs that do not have an edge-$3$-colouring, Ph.D. Thesis, Department of Combinatorics and Optimization, University of Waterloo, Waterloo, Canada, 1984.

\bibitem{ChudSeyClawFreeChar} M. Chudnovsky, P. Seymour, The structure of claw-free graphs. In Surveys in combinatorics 2005, London Math. Soc. Lecture Note Ser. 327, pages 153--171. Cambridge Univ. Press, Cambridge, 2005.

\bibitem{LucaAKCE2020} L. Ferrarini, G. Mazzuoccolo, V.V. Mkrtchyan, Normal $5$-edge-colorings of a family of Loupekhine snarks, AKCE International Journal of Graphs and Combinatorics 17(3), (2020), 720--724.

\bibitem{Fulkerson} D.R. Fulkerson, Blocking and anti-blocking pairs of polyhedra, Math. Programming 1 (1971), 168--194.

\bibitem{HaggSteff2013} J. H\"{a}gglund, E. Steffen, Petersen-colorings and some families of snarks, Ars Mathematica Contemporanea 7 (2014), 161--173.

\bibitem{HakobyanAMC} A. Hakobyan, V. Mkrtchyan, $S_{12}$ and $P_{12}$-colorings of cubic graphs, Ars Mathematica Contemporanea 17(2), (2019), 431--445.

\bibitem{HakobyanAUJC} A. Hakobyan, V. Mkrtchyan, On Sylvester Colorings of Cubic Graphs, Australasian Journal of Combinatorics 72(3), (2018), 472--491.

\bibitem{holyer:1981} I. Holyer, The NP-Completeness of Edge-coloring, SIAM J. Comp. 10(4), 718--720, (1981).

\bibitem{HolySkoJCTB2004} F. Holyord, M. \v{S}koviera, Colouring of cubic graphs by Steiner triple systems, J. Comb. Theory, Ser. B 91, (2004), 57--66.

\bibitem{Jaeger1975} F. Jaeger, On nowhere-zero flows in multigraphs, in ``Proceedings, Fifth British Combinatorial Conference, Aberdeen, 1975." Congressus Numerantium XV, Utilitas Mathematica Winnipeg, 373--378.

\bibitem{Jaeger1979} F. Jaeger, Flows and generalized coloring theorems in graphs, J. Comb. Theory, Ser. B 26, (1979), 205--216. 

\bibitem{Jaeger1985} F. Jaeger, On five-edge-colorings of cubic graphs and nowhere-zero flow problems, Ars Combinatoria, 20-B, (1985), 229--244.

\bibitem{Jaeger1988} F. Jaeger, Nowhere-zero flow problems, Selected topics in graph theory, 3, Academic Press, San Diego, CA, 1988, pp. 71--95.

\bibitem{JaegerSwart1980} F. Jaeger, T. Swart, Problem session, Ann. Discrete Math. 9, (1980), 305.


\bibitem{KKN2005} T. Kaiser, D. Kr\'{a}l, S. Norine, Unions of perfect matchings in cubic graphs. Electron. Notes Discret. Math. 22, (2005), 341--345.

\bibitem{KSSidma2008} T. Kaiser, R. \v{S}krekovski, Cycles Intersecting Edge-Cuts of Prescribed Sizes, SIAM Journal on Discrete Mathematics 22(3), (2008), 861--874.

\bibitem{Mac22} J. Karab\'{a}\v{s}, E. M\'{a}\v{c}ajov\'{a}, R. Nedela, M. \v{S}koviera, Girth, oddness, and colouring defect of snarks, Disc. Math. 345(11), (2022), 113040.

\bibitem{KMZ22} F. Kardo\v{s}, E. M\'{a}\v{c}ajov\'{a}, J.-P. Zerafa, Disjoint odd circuits in a bridgeless cubic graph can be quelled by a single perfect matching, J. Combinatorial Theory B 160, 2023, 1--14.

\bibitem{Kochol96} M. Kochol, Snarks without small cycles, Journal of Combinatorial Theory, Ser. B, 67, 34--47, (1996).

\bibitem{Lov1987} L. Lov\'{a}sz, Matching structure and the matching lattice, Journal of Comb. Th. B 43(2), (1987), pp. 187--222.

\bibitem{Lovasz} L. Lov\'asz, M.D. Plummer, Matching Theory, Annals of Discrete Math. 29, North Holland, 1986.

\bibitem{SidmaPerfects} Y. Ma, D. Mattiolo, E. Steffen, I.H. Wolf, Pairwise Disjoint Perfect Matchings in $r$-Edge-Connected $r$-Regular Graphs, SIAM Journal on Discrete Mathematics 37(3), (2023), 1548--1565.

\bibitem{rGraphHColorings} Y. Ma, D. Mattiolo, E. Steffen, I.H. Wolf, Sets of $r$-graphs that color all $r$-graphs, arXiv:2305.08619, 2023.

\bibitem{MMBergeFulkersonCyclically5} E. M\'{a}\v{c}ajov\'{a}, G. Mazzuoccolo, Reduction of the Berge-Fulkerson conjecture to cyclically 5-edge-connected snarks, Proceedings of the American Mathematical Society 148(11), (2020), 1.

\bibitem{MacSko21} E. M\'{a}\v{c}ajov\'{a}, M. \v{S}koviera, Superposition of snarks revisited, European J. Combin. 91 (2021), Article 103220.

\bibitem{MacSko2007} E. M\'{a}\v{c}ajov\'{a}, M. \v{S}koviera, Constructing Hypohamiltonian Snarks with Cyclic Connectivity 5 and 6, The Electronic Journal of Combinatorics 14, (2007), \#R18.

\bibitem{MMM2021} D. Mattiolo, G. Mazzuoccolo, V.V. Mkrtchyan, On sublinear approximations for the Petersen coloring conjecture, Bulletin of ICA 92, (2021), 78--90. 

\bibitem{Mazz11} G. Mazzuoccolo, The equivalence of two conjectures of Berge and Fulkerson, J. Graph Theory 68(2), 125--128, (2011)

\bibitem{Mazz13} G. Mazzuoccolo, An Upper Bound for the Excessive Index of an $r$-Graph, J. Graph Theory 73(4): 377--385, (2013)

\bibitem{Mazz2013} G. Mazzuoccolo, New conjectures on perfect matchings in cubic graphs, Electronic Notes in Discrete Mathematics 40, 235--238.

\bibitem{MM2020DAM} G. Mazzuoccolo, V.V. Mkrtchyan, Normal $6$-edge-colorings of some bridgeless cubic graphs, Discrete Applied Mathematics 277, (2020), 252--262

\bibitem{MM2020JGT} G. Mazzuoccolo, V.V. Mkrtchyan, Normal edge-colorings of cubic graphs, Journal of Graph Theory 94(1), (2020), 75--91.

\bibitem{HcoloringsRegulars} G. Mazzuoccolo, G. Tabarelli, J.P. Zerafa, On the existence of graphs which can colour every regular graph, Discret. Appl. Math. 337, (2023), 246--256.

\bibitem{Mkrt2013} V. Mkrtchyan, A remark on the Petersen coloring conjecture of Jaeger, Australasian J. Comb. 56(2013), pp. 145--151.



\bibitem{sang-il_oum:2011} S.-il~Oum, Perfect matchings in claw-free cubic graphs, The Electronic Journal of Combinatorics 18(1), 2011.

\bibitem{PirSerSkr} F. Pirot, J.S. Sereni, R. \v{S}krekovski, Variations on the Petersen Colouring Conjecture, The Electronic Journal of Combinatorics 27(1), (2020), \#P1.81.


\bibitem{Preiss1981} M. Preissmann, Sur les colorations des aretes des graphes cubiques, These de $3$-eme cycle, Grenoble (1981).

\bibitem{Rizzi1999} R. Rizzi, Indecomposable $r$-graphs and some other counterexamples, J. Graph Theory 32(1), 1-15 (1999)

\bibitem{Samal2011} R. \v{S}\'{a}mal, New approach to Petersen coloring, Elec. Notes in Discr. Math. 38 (2011), 755--760.

\bibitem{Samal2016} R. \v{S}\'{a}mal, personal communication to G. Mazzuoccolo and V.V. Mkrtchyan, 2016.

\bibitem{Sch} T. Sch\"onberger, Ein Beweis des Petersenschen Graphensatzes, Acta Scientia Mathematica Szeged 7, (1934), 51--57.

\bibitem{SedlarArXiv2023} J. Sedlar, R. \v{S}krekovski, Normal 5-edge coloring of some more snarks superpositioned by the Petersen graph, (2023), arXiv:2312.08739.

\bibitem{SedlarEuJC2024} J. Sedlar, R. \v{S}krekovski, Normal 5-edge-coloring of some snarks superpositioned by Flower snarks, European Journal of Combinatorics 122, (2024), 104038.

\bibitem{Seymour} P. D. Seymour, On multicolourings of cubic graphs, and conjectures of Fulkerson and Tutte. Proc. London Math. Soc. 38 (3), 423--460, 1979.

\bibitem{steffen:1998} E.~Steffen, Classifications and characterizations of snarks, Discrete Mathematics 188, 183--203, (1998).

\bibitem{SteffenHypoHamiltonian} E. Steffen, On bicritical snarks, Mathematica Slovaca 51, (2001), 141--150.

\bibitem{steffen:2004} E.~Steffen, Measurements of edge-uncolorability, Discrete Mathematics 280, 191--214, (2004).

\bibitem{stiebitz:2012} M.~Stiebitz, D.~Scheide, B.~Toft, L.~M. Favrholdt, Graph Edge Coloring, John Wiley and Sons, (2012).

\bibitem{Thom20} C. Thomassen, Factorizing regular graphs, J. Combin. Theory Ser. B, 141 (2020), 343--351.

\bibitem{vizing:1964} V.~Vizing, On an estimate of the chromatic class of a $p$-graph, Diskret Analiz (3), 25--30, (1964).

\bibitem{west:1996} D.~West, Introduction to Graph Theory, Prentice-Hall, Englewood Cliffs, (1996).

\bibitem{Zhang1997} C.-Q. Zhang, Integer flows and cycle covers of graphs, Marcel Dekker, Inc., New York Basel Hong Kong, 1997.



\end{thebibliography}

\end{document}